\documentclass[a4paper, oneside]{article}
\usepackage{enumitem}
\usepackage{lmodern}
\usepackage{amssymb,amsthm,amsmath}
\usepackage[utf8]{inputenc}
\usepackage{verbatim}
\usepackage{tikz}
\usepackage{authblk}
\usepackage{cite}
\usetikzlibrary{decorations.pathreplacing}

\theoremstyle{definition}
\newtheorem{definition}{Definition}[section]
\newtheorem{lemma}[definition]{Lemma}

\newtheorem{theorem}[definition]{Theorem}

\newtheorem{corollary}[definition]{Corollary}
\newtheorem{remark}[definition]{Remark}
\newtheorem{problem}[definition]{Problem}

\newcommand{\N}{\mathbb{N}}
\newcommand{\Npos}{{\mathbb{N}_+}}
\newcommand{\Z}{\mathbb{Z}}
\newcommand{\R}{\mathbb{R}}
\newcommand{\Rpos}{\mathbb{R}_{>0}}
\newcommand{\abs}[1]{\vert #1 \vert} 
\newcommand{\id}{\mathrm{id}} 
\newcommand{\lfast}{{\leftarrow}}
\newcommand{\rfast}{{\rightarrow}}
\newcommand{\lslow}{{\swarrow}}
\newcommand{\rslow}{{\searrow}}
\newcommand{\wall}{{\|}}
\newcommand{\ord}{{\mathcal{O}}}
\newcommand{\digs}{\Sigma} 
\newcommand{\Mul}{\Pi} 
\newcommand{\mul}{g} 

\newcommand{\rfig}[1]{Figure~\ref{#1}} 

\newcommand{\rlem}[1]{Lemma~\ref{#1}}

\usetikzlibrary{decorations.pathreplacing}
\usetikzlibrary{automata}
\usetikzlibrary{positioning}

\DeclareMathOperator{\fractional}{frac} 
\DeclareMathOperator{\cyl}{Cyl} 
\DeclareMathOperator{\config}{config} 
\DeclareMathOperator{\real}{real} 
\DeclareMathOperator{\integ}{int} 

\begin{document}

\title{On computing the Lyapunov exponents of reversible cellular automata\thanks{The work was partially supported by the Academy of Finland grant 296018 and by the Vilho, Yrjö and Kalle Väisälä Foundation}}

\author{Johan Kopra}

\affil{Department of Mathematics and Statistics, \\FI-20014 University of Turku, Finland}
\affil{jtjkop@utu.fi}

\date{}

\maketitle

\setcounter{page}{1}

\begin{abstract}
We consider the problem of computing the Lyapunov exponents of reversible cellular automata (CA). We show that the class of reversible CA with right Lyapunov exponent $2$ cannot be separated algorithmically from the class of reversible CA whose right Lyapunov exponents are at most $2-\delta$ for some absolute constant $\delta>0$. Therefore there is no algorithm that, given as an input a description of an arbitrary reversible CA $F$ and a positive rational number $\epsilon>0$, outputs the Lyapunov exponents of $F$ with accuracy $\epsilon$. We also compute the average Lyapunov exponents (with respect to the uniform measure) of the CA that perform multiplication by $p$ in base $pq$ for coprime $p,q>1$.
\end{abstract}

\providecommand{\keywords}[1]{\textbf{Keywords:} #1}
\noindent\keywords{Cellular automata, Lyapunov exponents, Reversible computation, Computability}

\section{Introduction}

A cellular automaton (CA) is a model of parallel computation consisting of a uniform (in our case one-dimensional) grid of finite state machines, each of which receives input from a finite number of neighbors. All the machines use the same local update rule to update their states simultaneously at discrete time steps. The following question of error propagation arises naturally: If one changes the state at some of the coordinates, then how long does it take for this change to affect the computation at other coordinates that are possibly very far away? Lyapunov exponents provide one tool to study the asymptotic speeds of error propagation in different directions. The concept of Lyapunov exponents originally comes from the theory of differentiable dynamical systems, and the discrete variant of Lyapunov exponents for CA was originally defined in~\cite{She92}.

The Lyapunov exponents of a cellular automaton $F$ are interesting also when one considers $F$ as a topological dynamical system, because they can be used to give an upper bound for the topological entropy of $F$ \cite{Tis00}. It is previously known that the entropy of one-dimensional cellular automata is uncomputable \cite{HKC92} (and furthermore from \cite{GZ12} it follows that there exists a single cellular automaton whose entropy is uncomputable), which gives reason to suspect that also the Lyapunov exponents are uncomputable in general.

The uncomputability of Lyapunov exponents is easy to prove for (not necessarily reversible) cellular automata by using the result from \cite{Kari92} which says that nilpotency of cellular automata with a spreading state is undecidable. We will prove the more specific claim that the Lyapunov exponents are uncomputable even for reversible cellular automata. In the context of proving undecidability results for reversible CA one cannot utilize undecidability of nilpotency for non-reversible CA. An analogous decision problem, the (local) immortality problem, has been used to prove undecidability results for reversible CA \cite{Luk10}. We will use in our proof the undecidability of a variant of the immortality problem, which in turn follows from the undecidability of the tiling problem for $2$-way deterministic tile sets. This result has been published previously in the proceedings of UCNC 2019~\cite{Kop19}.

In the other direction, there are interesting classes of cellular automata whose Lyapunov exponents have been computed. In \cite{DMM03} a closed formula for the Lyapunov exponents of linear one-dimensional cellular automata is given. We present results concerning the family of multiplication automata that simulate multiplication by $p$ in base $pq$ for coprime $p,q>1$. It is easy to find out their Lyapunov exponents and it is more interesting to consider a measure theorical variant, the average Lyapunov exponent. We compute the average Lyapunov exponents (with respect to the uniform measure) for these automata. This computation is originally from the author's Ph.D. thesis~\cite{KopDiss}.

\section{Preliminaries}

For sets $A$ and $B$  we denote by $B^A$ the collection of all functions from $A$ to $B$.

A finite set $A$ of \emph{letters} or \emph{symbols} is called an \emph{alphabet}. The set $A^{\Z}$ is called a \emph{configuration space} or a \emph{full shift} and its elements are called \emph{configurations}. An element $x\in A^{\Z}$ is interpreted as a bi-infinite sequence and we denote by $x[i]$ the symbol that occurs in $x$ at position $i$. A \emph{factor} of $x$ is any finite sequence $x[i] x[i+1]\cdots x[j]$ where $i,j\in\Z$, and we interpret the sequence to be empty if $j<i$. Any finite sequence $w=w[1] w[2]\cdots w[n]$ (also the empty sequence, which is denoted by $\epsilon$) where $w[i]\in A$ is a \emph{word} over $A$. If $w\neq\epsilon$, we say that $w$ \emph{occurs} in $x$ at position $i$ if $x[i]\cdots x[i+n-1]=w[1]\cdots w[n]$ and we denote by $w^\Z\in A^\Z$ the configuration in which $w$ occurs at all positions of the form $in$ ($i\in\Z$). The set of all words over $A$ is denoted by $A^*$, and the set of non-empty words is $A^+=A^*\setminus\{\epsilon\}$. More generally, for $L,K\subseteq A^*$ we denote $LK=\{w_1w_2\mid w_1\in L, w_2\in K\}$ and $L^*=\{w_1\cdots w_n\mid n\geq 0, w_i\in L\}$. If $\epsilon\notin L$, define $L^+=L^*\setminus\{\epsilon\}$ and if $\epsilon\in L$, define $L^+=L^*$. The set of words of length $n$ is denoted by $A^n$. For a word $w\in A^*$, $\abs{w}$ denotes its length, i.e. $\abs{w}=n\iff w\in A^n$.

If $A$ is an alphabet and $C$ is a countable set, then $A^C$ becomes a compact metrizable topological space when endowed with the product topology of the discrete topology of $A$ (in particular a set $S\subseteq A^C$ is compact if and only if it is closed). In our considerations $C=\Z$ or $C=\Z^2$. We define the \emph{shift} $\sigma:A^\Z\to A^\Z$ by $\sigma(x)[i]=x[i+1]$ for $x\in A^\Z$, $i\in\Z$, which is a homeomorphism. We say that a closed set $X\subseteq A^\Z$ is a \emph{subshift} if $\sigma(X)=X$. Any $w\in A^+$ and $i\in\Z$ determine a \emph{cylinder} of $X$
\[\cyl_X(w,i)=\{x\in X\mid w \mbox{ occurs in }x\mbox{ at position }i\}.\]
Every cylinder is an open set of $X$ and the collection of all cylinders
\[\mathcal{C}_X=\{\cyl_X(w,i)\mid w\in A^+, i\in\Z\}\]
form a basis for the topology of $X$.

Occasionally we consider configuration spaces $(A_1\times A_2)^\Z$ and then we may write $(x_1,x_2)\in (A_1\times A_2)^\Z$ where $x_i\in A_i^\Z$ using the natural bijection between the sets $A_1^\Z\times A_2^\Z$ and $(A_1\times A_2)^\Z$. We may use the terminology that $x_1$ is on the upper layer or on the $A_1$-layer, and similarly that $x_2$ is on the lower layer or on the $A_2$-layer.

\begin{definition}\label{cadef}
Let $X\subseteq A^\Z$ be a subshift. We say that the map $F:X\to X$ is a \emph{cellular automaton} (or a CA) on $X$ if there exist integers $m\leq a$ (memory and anticipation) and a \emph{local rule} $f:A^{a-m+1}\to A$ such that $F(x)[i]=f(x[i+m],\dots,x[i],\dots,x[i+a])$. If we can choose $f$ so that $-m=a=r\geq 0$, we say that $F$ is a radius-$r$ CA and if we can choose $m=0$ we say that $F$ is a \emph{one-sided} CA. A one-sided CA with anticipation $1$ is called a radius-$\frac{1}{2}$ CA.\end{definition}

We can extend any local rule $f:A^{a-m+1}\to B$ to words $w=w[1]\cdots w[a-m+n]\in A^{a-m+n}$ with $n\in\Npos$ by $f(w)=u=u[1]\cdots u[n]$, where $u[i]= f(w[i],\dots,w[i+a-m])$.

The CA-functions on $X$ are characterized as those continuous maps on $X$ that commute with the shift \cite{Hed69}. We say that a CA $F:X\to X$ is \emph{reversible} if it is bijective. Reversible CA are homeomorphisms on $X$. The book \cite{LM95} is a standard reference for subshifts and cellular automata on them.

For a given reversible CA $F:X\to X$ and a configuration $x\in X$ it is often helpful to consider the \emph{space-time diagram} of $x$ with respect to $F$. Formally it is the map $\theta\in A^{\Z^2}$ defined by $\theta(i,-j)=F^j(x)[i]$: the minus sign in this definition signifies that time increases in the negative direction of the vertical coordinate axis. Informally the space-time diagram of $x$ is a picture which depicts elements of the sequence $(F^i(x))_{i\in\Z}$ in such a way that $F^{i+1}(x)$ is drawn below $F^i(x)$ for every $i$.

The definition of Lyapunov exponents is from \cite{She92,Tis00}. For a fixed subshift $X\subseteq A^\Z$ and for $x\in X$, $s\in\Z$, denote $W_{s}^{+}(x)=\{y\in X \mid \forall i\geq s: y[i]=x[i]\}$ and $W_{s}^{-}(x)=\{y\in X \mid \forall i\leq s: y[i]=x[i]\}$. Then for given cellular automaton $F:X\to X$, $x\in X$, $n\in\N$, define
\begin{flalign*}
&\Lambda_{n}^{+}(x,F)=\min\{s\geq 0\mid\forall 1\leq i\leq n: F^i(W_{-s}^{+}(x))\subseteq W_0^{+}(F^i(x))\} \\
&\Lambda_{n}^{-}(x,F)=\min\{s\geq 0\mid\forall 1\leq i\leq n: F^i(W_{s}^{-}(x))\subseteq W_0^{-}(F^i(x))\}.
\end{flalign*}

These have shift-invariant versions $\overline{\Lambda}_n^{\pm}(x,F)=\max_{i\in\Z} \Lambda_n^{\pm}(\sigma^i(x),F)$. The quantities
\[\lambda^{+}(x,F)=\lim_{n\to\infty}\frac{\overline{\Lambda}_n^{+}(x,F)}{n},\qquad \lambda^{-}(x,F)=\lim_{n\to\infty}\frac{\overline{\Lambda}_n^{-}(x,F)}{n}\]
are called (when the limits exist) respectively the right and left \emph{Lyapunov exponents of} $x$ (with respect to $F$).

A global version of these are the quantities
\[\lambda^{+}(F)=\lim_{n\to\infty}\max_{x\in X}\frac{\Lambda_n^{+}(x,F)}{n},\qquad \lambda^{-}(F)=\lim_{n\to\infty}\max_{x\in X}\frac{\Lambda_n^{-}(x,F)}{n}\]
that are called respectively the right and left \emph{(maximal) Lyapunov exponents} of $F$. These limits exist by an application of Fekete's subadditive lemma (e.g. Lemma 4.1.7 in \cite{LM95}).

Occasionally we consider cellular automata from the measure theoretical point of view. For a subshift $X$ we denote by $\Sigma(\mathcal{C}_X)$ the \emph{sigma-algebra} generated by the collection of cylinders $\mathcal{C}_X$. It is the smallest collection of subsets of $X$ which contains all the elements of $\mathcal{C}_X$ and which is closed under complement and countable unions. A \emph{measure} on $X$ is a countably additive function $\mu:\Sigma(\mathcal{C}_X)\to[0,1]$ such that $\mu(X)=1$, i.e. $\mu(\bigcup_{i=0}^\infty A_i)=\sum_{i=0}^\infty\mu(A_i)$ whenever all $A_i\in \Sigma(\mathcal{C}_X)$ are pairwise disjoint. A measure $\mu$ on $X$ is completely determined by its values on cylinders. We say that a cellular automaton $F:X\to X$ is \emph{measure preserving} if $\mu(F^{-1}(S))=\mu(S)$ for all $S\in\Sigma(\mathcal{C})$.

On full shifts $A^\Z$ we are mostly interested in the \emph{uniform measure} determined by $\mu(\cyl(w,i))=\abs{A}^{-\abs{w}}$ for $w\in A^+$ and $i\in\Z$. By Theorem 5.4 in \cite{Hed69} any surjective CA $F:A^\Z\to A^\Z$ preserves this measure. For more on measure theory of cellular automata, see \cite{Piv09}.

Measure theoretical variants of Lyapunov exponents are defined as follows. Given a measure $\mu$ on $X$ and for $n\in\N$, let $I_{n,\mu}^{+}(F)=\int_{x\in X}\Lambda_n^{+}(x,F) d\mu$ and $I_{n,\mu}^{-}(F)=\int_{x\in X}\Lambda_n^{-}(x,F) d\mu$. Then the quantities
\[I_{\mu}^{+}(F)=\liminf_{n\to\infty}\frac{I_{n,\mu}^{+}(F)}{n},\qquad I_{\mu}^{-}(F)=\liminf_{n\to\infty}\frac{I_{n,\mu}^{-}(F)}{n}\]
are called respectively the right and left \emph{average Lyapunov exponents of} $F$ (with respect to the measure $\mu$).

We will write $W_{s}^{\pm}(x)$, $\Lambda_n^{\pm}(x)$, $\lambda^{+}(x)$, $I_{n,\mu}^{+}$ and $I_{\mu}^{+}$ when $X$ and $F$ are clear by the context.

\section{Tilings and Undecidability}

In this section we recall the well-known connection between cellular automata and tilings on the plane. We use this connection to prove an auxiliary undecidability result for reversible cellular automata.

\begin{definition}A \emph{Wang tile} is formally a function $t:\{N,E,S,W\}\to C$ whose value at $I$ is denoted by $t_I$. Informally, a Wang tile $t$ should be interpreted as a unit square with edges colored by elements of $C$. The edges are called \emph{north}, \emph{east}, \emph{south} and \emph{west} in the natural way, and the colors in these edges of $t$ are  $t_N,t_E,t_S$ and $t_W$ respectively. A \emph{tile set} is a finite collection of Wang tiles.\end{definition}

\begin{definition}A tiling over a tile set $T$ is a function $\eta\in T^{\Z^2}$ which assigns a tile to every integer point of the plane. A tiling $\eta$ is said to be valid if neighboring tiles always have matching colors in their edges, i.e. for every $(i,j)\in\Z^2$ we have $\eta(i,j)_N=\eta(i,j+1)_S$ and $\eta(i,j)_E=\eta(i+1,j)_W$. If there is a valid tiling over $T$, we say that $T$ \emph{admits} a valid tiling.\end{definition}

We say that a tile set $T$ is NE-deterministic if for every pair of tiles $t,s\in T$ the equalities $t_N=s_N$ and $t_E=s_E$ imply $t=s$, i.e. a tile is determined uniquely by its north and east edge. A SW-deterministic tile set is defined similarly. If $T$ is both NE-deterministic and SW-deterministic, it is said to be \emph{2-way deterministic}.

The \emph{tiling problem} is the problem of determining whether a given tile set $T$ admits a valid tiling.

\begin{theorem}{\cite[Theorem 4.2.1]{Luk10}}\label{tiling} The tiling problem is undecidable for 2-way deterministic tile sets.\end{theorem}

\begin{definition}Let $T$ be a 2-way deterministic tile set and $C$ the collection of all colors which appear in some edge of some tile of $T$. $T$ is \emph{complete} if for each pair $(a,b)\in C^2$ there exist (unique) tiles $t,s\in T$ such that $(t_N,t_E)=(a,b)$ and $(s_S,s_W)=(a,b)$.\end{definition}

A 2-way deterministic tile set $T$ can be used to construct a complete tile set. Namely, let $C$ be the set of colors which appear in tiles of $T$, let $X\subseteq C\times C$ be the set of pairs of colors which do not appear in the northeast of any tile and let $Y\subseteq C\times C$ be the set of pairs of colors which do not appear in the southwest of any tile. Since $T$ is 2-way deterministic, there is a bijection $p:X\to Y$. Let $T^\complement$ be the set of tiles formed by matching the northeast corners $X$ with the southwest corners $Y$ via the bijection $p$. Then the tile set $A=T\cup T^\complement$ is complete.

Every complete 2-way deterministic tile set $A$ determines a local rule $f:A^2\to A$ defined by $f(a,b)=c\in A$, where $c$ is the unique tile such that $a_S=c_N$ and $b_W=c_E$. This then determines a reversible CA $F:A^\Z\to A^\Z$ with memory $0$ by $F(x)[i]=f(x[i],x[i+1])$ for $x\in A^\Z$, $i\in\Z$. The space-time diagram of a configuration $x\in A^\Z$ corresponds to a valid tiling $\eta$ via $\theta(i,-j)=F^j(x)[i]=\eta(i,-i-j)$, i.e. configurations $F^j(x)$ are diagonals of $\eta$ going from northwest to southeast and the diagonal corresponding to $F^{j+1}(x)$ is below the diagonal corresponding to $F^j(x)$.

\begin{definition}
A cellular automaton $F:A^\Z\to A^\Z$ is $(p,q)$-locally immortal ($p,q\in\N$) with respect to a subset $B\subseteq A$ if there exists a configuration $x\in A^\Z$ such that $F^{iq+j}(x)[ip]\in B$ for all $i\in\Z$ and $0\leq j\leq q$. Such a configuration $x$ is a $(p,q)$-witness.
\end{definition}
Generalizing the definition in \cite{Luk10}, we call the following decision problem the $(p,q)$-\emph{local immortality problem}: given a reversible CA $F:A^\Z\to A^\Z$ and a subset $B\subseteq A$, find whether $F$ is $(p,q)$-locally immortal with respect to $B$.

\begin{theorem}{\cite[Theorem 5.1.5]{Luk10}}\label{01immortal} The $(0,1)$-local immortality problem is undecidable for reversible CA.\end{theorem}

We now adapt the proof of Theorem \ref{01immortal} to get the following result, which we will use in the proof of Theorem~\ref{TheoremLyapSofic}.

\begin{lemma}\label{local}The $(1,5)$-local immortality problem is undecidable for reversible radius-$\frac{1}{2}$ CA.\end{lemma}
\begin{proof}We will reduce the problem of Theorem~\ref{tiling} to the $(1,5)$-local immortality problem. Let $T$ be a 2-way deterministic tile set and construct a complete tile set $T\cup T^\complement$ as indicated above. Then also $A_1=(T\times T_1)\cup(T^\complement\times T_2)$ ($T_1$ and $T_2$ as in \rfig{arrows}) is a complete tile set.\footnote{The arrow markings are used as a shorthand for some coloring such that the heads and tails of the arrows in neighboring tiles match in a valid tiling.} We denote the blank tile of the set $T_1$ by $t_b$ and call the elements of $R=A_1\setminus(T\times\{t_b\})$ arrow tiles. As indicated above, the tile set $A_1$ determines a reversible radius-$\frac{1}{2}$ CA $G_1:A_1^\Z\to A_1^\Z$.

\begin{figure}
\centering
\begin{tikzpicture}
\def\rectanglepath{-- ++(1.5cm,0cm)-- ++(0cm,1.5cm)-- ++(-1.5cm,0cm) -- cycle}
\draw (0,0) \rectanglepath;
\draw (2,0) \rectanglepath; \draw [->] (3.5cm,0.75cm) -- (2.05cm,0.75cm);
\draw (4,0) \rectanglepath; \draw [->] (4.75cm,1.5cm) -- (4.75cm,0.05cm);
\draw (6,0) \rectanglepath; \draw [->] (7.5cm,0.75cm) -- (6.05cm,0.75cm); \draw [->] (6.75cm,1.5cm) -- (6.75cm,0.05cm);

\draw (0,-2) \rectanglepath; \draw [->] (0.75,-1.25) -- (0.05,-1.25);
\draw (2,-2) \rectanglepath; \draw [->] (3.5,-1.25) -- (2.05,-1.25); \draw [->] (2.75,-1.25) -- (2.75,-1.95);
\draw (4,-2) \rectanglepath; \draw [->] (4.75cm,-0.5cm) -- (4.75cm,-1.95cm);
\draw (6,-2) \rectanglepath; \draw [->] (7.5cm,-1.25cm) -- (6.80cm,-1.25cm); \draw [->] (6.75cm,-0.5cm) -- (6.75cm,-1.20cm);
\end{tikzpicture}
\caption{The tile sets $T_1$ (first row) and $T_2$ (second row). These are originally from \cite{Luk10} (up to a reflection with respect to the northwest - southeast diagonal).}
\label{arrows}
\end{figure}
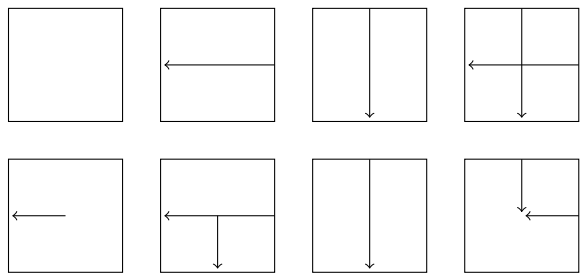

Let $A_2=\{0,1,2\}$. Define $A=A_1\times A_2$ and natural projections $\pi_i:A\to A_i$, $\pi_i(a_1,a_2)=a_i$ for $i\in\{1,2\}$. By extension we say that $a\in A$ is an arrow tile if $\pi_1(a)\in R$. Let $G:A^\Z\to A^\Z$ be defined by $G(c,e)=(G_1(c),e)$ where $c\in A_1^\Z$ and $e\in A_2^\Z$, i.e. $G$ simulates $G_1$ in the upper layer. We construct involutive CA $J_1$, $J_2$ and $H$ of memory $0$ with local rules $j_1:A_2\to A_2$, $j_2:A_2^2\to A_2$ and $h:(A_1\times A_2)\to (A_1\times A_2)$ respectively defined by
\begin{flalign*}
&\begin{array}{c}
j_1(0)=0 \\
j_1(1)=2 \\
j_1(2)=1
\end{array}
\qquad
j_2(a,b)=
\left\{
\begin{array}{l}
1 \mbox{ when } (a,b)=(0,2) \\
0 \mbox{ when } (a,b)=(1,2) \\
a \mbox{ otherwise }
\end{array}
\right. \\
&h((a,b))=
\left\{
\begin{array}{l}
(a,1) \mbox{ when } a\in R \mbox{ and } b=0 \\
(a,0) \mbox{ when } a\in R \mbox{ and } b=1 \\
(a,b) \mbox{ otherwise.}
\end{array}
\right.
\end{flalign*}

If $\id:A_1^\Z\to A_1^\Z$ is the identity map, then $J=(\id\times J_2)\circ (\id\times J_1)$ is a CA on $A^\Z=(A_1\times A_2)^\Z$. We define the radius-$\frac{1}{2}$ automaton $F=H\circ J\circ G:A^\Z\to A^\Z$ and select $B=(T\times\{t_b\})\times\{0\}$. We will show that $T$ admits a valid tiling if and only if $F$ is $(1,5)$-locally immortal with respect to $B$.

Assume first that $T$ admits a valid tiling $\eta$. Then by choosing $x\in A^\Z$ such that $x[i]=((\eta(i,-i),t_b),0)\in A_1\times A_2$ for $i\in\Z$ it follows that $F^j(x)[i]\in B$ for all $i,j\in\Z$ and in particular that $x$ is a $(1,5)$-witness. 

Assume then that $T$ does not admit any valid tiling and for a contradiction assume that $x$ is a $(1,5)$-witness. Let $\theta$ be the space-time diagram of $x$ with respect to $F$. Since $x$ is a $(1,5)$-witness, it follows that $\theta(i,-j)\in B$ whenever $(i,-j)\in N$, where $N=\{(i,-j)\in \Z^2\mid 5i\leq j\leq 5(i+1)\}$. There is a valid tiling $\eta$ over $A_1$ such that $\pi_1(\theta(i,j))=\eta(i,j-i)$ for $(i,j)\in\Z^2$, i.e. $\eta$ can be recovered from the upper layer of $\theta$ by applying a suitable linear transformation on the space-time diagram. In drawing pictorial representations of $\theta$ we want that the heads and tails of all arrows remain properly matched in neighboring coordinates, so we will use tiles with ``bent'' labelings, see \rfig{bentarrows}. Since $T$ does not admit valid tilings, it follows by a compactness argument that $\eta(i,j)\notin T\times T_1$ for some $(i,j)\in D$ where $D=\{(i,j)\in\Z^2\mid j>-6i \}$ and in particular that $\eta(i,j)$ is an arrow tile. Since $\theta$ contains a ``bent'' version of $\eta$, it follows that $\theta(i,j)$ is an arrow tile for some $(i,j)\in E$, where $E=\{(i,j)\in\Z^2\mid j>-5i\}$ is a ``bent'' version of the set $D$. In \rfig{stdArr} we present the space-time diagram $\theta$ with arrow markings of tiles from $T_1$ and $T_2$ replaced according to the \rfig{bentarrows}. In \rfig{stdArr} we have also marked the sets $N$ and $E$. Other features of the figure become relevant in the next paragraph.

\begin{figure}[ht]
\centering
\begin{tikzpicture}
\def\rectanglepath{-- ++(1.5cm,0cm)-- ++(0cm,1.5cm)-- ++(-1.5cm,0cm) -- cycle}
\draw (0,0) \rectanglepath;
\draw (2,0) \rectanglepath; \draw [->] (3.5cm,1.5cm) -- (2.05cm,0.05cm);
\draw (4,0) \rectanglepath; \draw [->] (4.75cm,1.45cm) -- (4.75cm,0cm);
\draw (6,0) \rectanglepath; \draw [->] (7.5cm,1.5cm) -- (6.05cm,0.05cm); \draw [->] (6.75cm,1.45cm) -- (6.75cm,0cm);

\draw (0,-2) \rectanglepath; \draw [->] (0.75,-1.25) -- (0.05,-1.95);
\draw (2,-2) \rectanglepath; \draw [->] (3.5,-0.5) -- (2.05,-1.95); \draw [->] (2.75,-1.25) -- (2.75,-1.95);
\draw (4,-2) \rectanglepath; \draw [->] (4.75cm,-0.5cm) -- (4.75cm,-1.95cm);
\draw (6,-2) \rectanglepath; \draw [->] (7.5cm,-0.5cm) -- (6.85cm,-1.20cm); \draw [->] (6.75cm,-0.5cm) -- (6.75cm,-1.20cm);
\end{tikzpicture}
\caption{The tile sets $T_1$ and $T_2$ presented in a ``bent'' form.}
\label{bentarrows}

\vspace{3.00mm}

\begin{tikzpicture}[scale=0.6]

\begin{scope}[yscale=-1,xscale=1]
\def\band{-- ++(0,5)-- ++(1,0) ++(0,-6) -- ++(0,5)-- ++(1,0)}
\def\rectanglepath{-- ++(1cm,0cm)-- ++(0cm,1cm)-- ++(-1cm,0cm) -- cycle}
\draw (0,1) \band; \draw (1,6) \band;
\draw [->] (2.5,0) -- (2.5,3);
\draw [->] (2.5,3) -- (2.5,4.4);
\draw [->] (6.5,0.5) -- (2.6,4.4);
\draw [->] (3.5,3.5) -- (3.5,6.4);
\draw [->] (7.5,2.5) -- (3.6,6.4);
\draw (5.5,1.5) -- (5.5,11);

\draw [decorate,decoration=brace] (0.5,1.5) -- (2.5,1.5) node[midway,yshift=10pt]{$d_1$};
\draw [decorate,decoration=brace] (2.5,4.5) -- (2.5,10.5) node[midway,xshift=10pt]{$d_2$};
\draw[thick] (2,2) \rectanglepath;
\draw[thick] (2,4) \rectanglepath;

\node at (1,5.5) {$N$};
\node at (4.5,7.5) {$E$};
\node at (3.9,2.5) {$\theta(p,q)$};
\node at (4.4,4.5) {$\theta(p,q-2)$};
\end{scope}
\end{tikzpicture}

\caption{The space-time diagram $\theta$ with ``bent'' arrow markings. An arrow tile $\theta(p,q-2)$ in $E$ with minimal horizontal and vertical distances to $N$ has been highlighted.}
\label{stdArr}
\end{figure}

The minimal distance between a tile in $N$ and an arrow tile in $E$ situated on the same horizontal line in $\theta$ is denoted by $d_1>0$. Then, among those arrow tiles in $E$ at horizontal distance $d_1$ from $N$, there is a tile with minimal vertical distance $d_2>0$ from $N$ (see \rfig{stdArr}). Fix $p,q\in\Z$ so that $\theta(p,q-2)$ is one such tile and in particular $(p-d_1,q-2),(p,q-2-d_2)\in N$. Then $\theta(p,q-j)$ contains an arrow for $-2\leq j \leq 2$, because if there is a $j\in[-2,2)$ such that $\theta(p,q-j)$ does not contain an arrow and $\theta(p,q-j-1)$ does, then $\theta(p,q-j-1)$ must contain one of the three arrows on the left half of \rfig{bentarrows}. These three arrows continue to the southwest, so then also $\theta(p-1,q-j-2)$ contains an arrow. Because $\theta(p',q')\in B$ for $(p',q')\in N$, it follows that $(p-1,q-j-2)\notin N$ and thus $(p-1,q-j-2)\in E$. Since $(p-d_1,q-2)\in N$, it follows that one of the $(p-d_1-1,q-j-2)$, $(p-d_1,q-j-2)$ and $(p-d_1+1,q-j-2)$ belong to $N$. Thus the horizontal distance of the tile $\theta(p-1,q-j-2)$ from the set $N$ is at most $d_1$, and is actually equal to $d_1$ by the minimality of $d_1$. Since $N$ is invariant under translation by the vector $-(1,-5)$, then from $(p,q-2-d_2)\in N$ it follows that $(p-1,q+3-d_2)\in N$ and that the vertical distance of the tile $\theta(p-1,q-j-2)$ from $N$ is at most $(q-j-2)-(q+3-d_2)\leq d_2-3$, contradicting the minimality of $d_2$. Similarly, $\theta(p-i,q-j)$ does not contain an arrow for $0<i\leq d_1,$ $-2\leq j\leq 2$ by the minimality of $d_1$ and $d_2$.

Now consider the $A_2$-layer of $\theta$. For the rest of the proof let $y=F^{-q}(x)$. Assume that $\pi_2(\theta(p-i,q))=\pi_2(y[p-i])$ is non-zero for some $i\geq 0$, $(p-i,q)\in E$, and fix the greatest such $i$, i.e. $\pi_2(y[s])=0$ for $s$ in the set
\[I_0=\{p'\in\Z\mid p'<p-i, (p',q)\in N\cup E\}.\]
We start by considering the case $\pi_2(y[p-i])=1$. Denote 
\[I_1=\{p'\in\Z\mid p'<p-i, (p',q-1)\in N\cup E\}\subseteq I_0.\]
From the choice of $(p,q)$ it follows that $\pi_1(\theta(s,q-1))=\pi_1(G(y)[s])$ are not arrow tiles for $s\in I_1$, and therefore we can compute step by step that
\begin{flalign*}
&\pi_2((\id\times J_1)(G(y))[p-i])=2, \quad\pi_2((\id\times J_1)(G(y))[s])=0\mbox{ for }s\in I_0\subseteq I_1, \\
&\pi_2(J(G(y))[p-(i+1)])=1, \quad\pi_2(J(G(y))[s])=0\mbox{ for }s\in I_1\setminus\{p-(i+1)\}, \\
&\pi_2(F(y))[p-(i+1)])=1, \quad\pi_2(F(y)[s])=0\mbox{ for }s\in I_1\setminus\{p-(i+1)\}
\end{flalign*}
and $\pi_2(\theta(p-(i+1),q-1))=1$. By repeating this argument inductively we see that the digit $1$ propagates to the lower left in the space-time diagram as indicated by \rfig{left} and eventually reaches $N$, a contradiction. If on the other hand $\pi_2(\theta(p-i,q))=2$, a similar argument shows that the digit $2$ propagates to the upper left in the space-time diagram as indicated by \rfig{left} and eventually reaches $N$, also a contradiction.

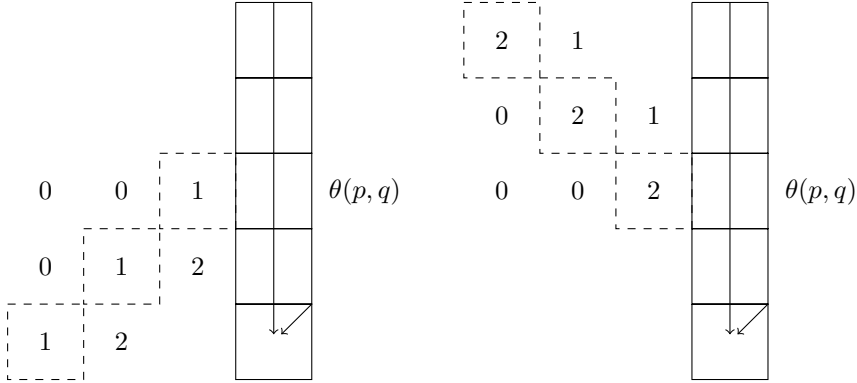
\begin{figure}
\centering
\begin{tikzpicture}
\begin{scope}[yscale=-1,xscale=1]
\def\rectanglepath{-- ++(1cm,0cm)-- ++(0cm,1cm)-- ++(-1cm,0cm) -- ++(0cm,-1cm)}
\def\fivesquares{\rectanglepath ++(0,1) \rectanglepath ++(0,1) \rectanglepath ++(0,1) \rectanglepath ++(0,1) \rectanglepath}
\def\fivelines{++(0.5,0) -- ++(0,1) -- ++(0,1) -- ++(0,1) -- ++(0,1) -- ++(0,0.4)}

\draw (3,0) \fivesquares; \draw [->] (3,0) \fivelines; \draw [->] (4,4) -- (3.6,4.4);
\draw[dashed] (2,2) \rectanglepath; \draw[dashed] (1,3) \rectanglepath; \draw[dashed] (0,4) \rectanglepath;
\node at (0.5,4.5) {$1$}; \node at (1.5,4.5) {$2$}; 
\node at (0.5,3.5) {$0$}; \node at (1.5,3.5) {$1$}; \node at (2.5,3.5) {$2$}; 
\node at (0.5,2.5) {$0$}; \node at (1.5,2.5) {$0$}; \node at (2.5,2.5) {$1$};
\node at (4.7,2.5) {$\theta(p,q)$};

\draw (9,0) \fivesquares; \draw [->] (9,0) \fivelines; \draw [->] (10,4) -- (9.6,4.4);
\draw[dashed] (8,2) \rectanglepath; \draw[dashed] (7,1) \rectanglepath; \draw[dashed] (6,0) \rectanglepath;
\node at (6.5,2.5) {$0$}; \node at (7.5,2.5) {$0$}; \node at (8.5,2.5) {$2$};
\node at (6.5,1.5) {$0$}; \node at (7.5,1.5) {$2$}; \node at (8.5,1.5) {$1$};
\node at (6.5,0.5) {$2$}; \node at (7.5,0.5) {$1$};
\node at (10.7,2.5) {$\theta(p,q)$};
\end{scope}
\end{tikzpicture}
\caption{Propagation of digits to the left of $\theta(p,q)$.}
\label{left}
\end{figure}

Assume then that $\pi_2(\theta(p-i,q))$ is zero whenever $i\geq 0$, $(p-i,q)\in E$. If $\pi_2(\theta(p+1,q))=\pi_2(y[p+1])\neq 1$, then $\pi_2((\id\times J_1)(G(y))[p+1])\neq 2$ and $\pi_2(J(G(y))[p])=0$. Since $\pi_1(\theta(p,q-1))$ is an arrow tile, it follows that $\pi_2(\theta(p,q-1))=\pi_2(H(J(G(y)))[p])=1$. The argument of the previous paragraph shows that the digit $1$ propagates to the lower left in the space-time diagram as indicated by the left side of \rfig{right} and eventually reaches $N$, a contradiction. 

Finally consider the case $\pi_2(\theta(p+1,q))=\pi_2(y[p+1])=1$. Then
\begin{flalign*}
&\pi_2(J(G(y))[p])\pi_2(J(G(y))[p+1])=12 \mbox{ and } \\
&\pi_2(F(y)[p])\pi_2(F(y)[p+1])=02.
\end{flalign*}
As in the previous paragraph we see that $\pi_2(\theta(p,q-2))=1$. This occurrence of the digit $1$ propagates to the lower left in the space-time diagram as indicated by the right side of \rfig{right} and eventually reaches $N$, a contradiction.

\begin{figure}
\centering
\begin{tikzpicture}
\begin{scope}[yscale=-1,xscale=1]
\def\rectanglepath{-- ++(1cm,0cm)-- ++(0cm,1cm)-- ++(-1cm,0cm) -- ++(0cm,-1cm)}
\def\threesquares{\rectanglepath ++(0,1) \rectanglepath ++(0,1) \rectanglepath}
\def\threelines{++(0.5,0) -- ++(0,1) -- ++(0,1) -- ++(0,0.5)}

\draw (3,1) \threesquares; \draw [->] (3,1) \threelines; \draw [->] (4,3) -- (3.6,3.4);
\draw[dashed] (2,3) \rectanglepath; \draw[dashed] (1,4) \rectanglepath;
\node at (1.5,4.5) {$1$}; \node at (2.5,4.5) {$2$};
\node at (1.5,3.5) {$0$}; \node at (2.5,3.5) {$1$}; \node[fill=white] at (3.5,3.5) {$2$};
\node at (1.5,2.5) {$0$}; \node at (2.5,2.5) {$0$}; \node[fill=white] at (3.5,2.5) {$1$};
\node at (1.5,1.5) {$0$}; \node at (2.5,1.5) {$0$}; \node[fill=white] at (3.5,1.5) {$0$}; \node at (4.5,1.5) {$\not 1$};
\node at (3.5,0.5) {$\theta(p,q)$};

\draw (9,1) \threesquares; \draw [->] (9,1) \threelines; \draw [->] (10,3) -- (9.6,3.4);
\draw[dashed] (8,4) \rectanglepath;
\node at (7.5,4.5) {$0$}; \node at (8.5,4.5) {$1$}; \node at (9.5,4.5) {$2$};
\node at (7.5,3.5) {$0$}; \node at (8.5,3.5) {$0$}; \node[fill=white] at (9.5,3.5) {$1$};
\node at (7.5,2.5) {$0$}; \node at (8.5,2.5) {$0$}; \node[fill=white] at (9.5,2.5) {$0$}; \node at (10.5,2.5) {$2$};
\node at (7.5,1.5) {$0$}; \node at (8.5,1.5) {$0$}; \node[fill=white] at (9.5,1.5) {$0$}; \node at (10.5,1.5) {$1$};
\node at (9.5,0.5) {$\theta(p,q)$};
\end{scope}
\end{tikzpicture}
\caption{Propagation of digits at $\theta(p,q)$.}
\label{right}
\end{figure}
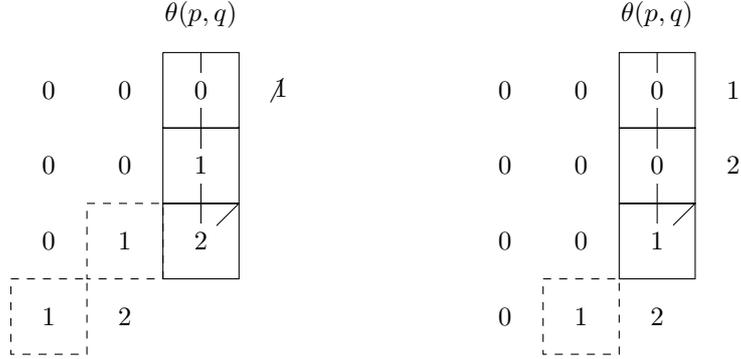

\end{proof}

\begin{remark}
It is possible that the $(p,q)$-local immortality problem is undecidable for reversible radius-$\frac{1}{2}$ CA whenever $p\in\N$ and $q\in\Npos$. We proved this in the case $(p,q)=(1,5)$ but for our purposes it is sufficient to prove this just for some $p>0$ and $q>0$. The important (seemingly paradoxical) part will be that for $(1,5)$-locally immortal radius-$\frac{1}{2}$ CA $F$ the ``local immortality'' travels to the right in the space-time diagram even though in reality there cannot be any information flow to the right because $F$ is one-sided. 
\end{remark}

\section{Uncomputability of Lyapunov exponents}

In this section we will prove our main result saying that there is no algorithm that can compute the Lyapunov exponents of a given reversible cellular automaton on a full shift to an arbitrary precision.

To achieve greater clarity we first prove this result in a more general class of subshifts. For the statement of the following theorem, we recall for completeness that a sofic shift $X\subseteq A^\Z$ is a subshift that can be represented as the set of labels of all bi-infinite paths on some labeled directed graph. This precise definition will not be of any particular importance, because the sofic shifts that we construct are of very specific form. We will appeal to the proof of the following theorem during the course of the proof of our main result. 

\begin{theorem}\label{TheoremLyapSofic}For reversible CA $F:X\to X$ on sofic shifts such that $\lambda^{+}(F)\in [0,\frac{5}{3}]\cup\{2\}$ it is undecidable whether $\lambda^{+}(F)\leq \frac{5}{3}$ or $\lambda^{+}(F)=2$.\end{theorem}
\begin{proof}We will reduce the decision problem of \rlem{local} to the present problem. Let $G:A_2^\Z\to A_2^\Z$ be a given reversible radius-$\frac{1}{2}$ cellular automaton and $B\subseteq A_2$ some given set. Let $A_1=\{0,\wall,\lfast,\rfast,\lslow,\rslow\}$ and define a sofic shift $Y\subseteq A_1^\Z$ as the set of those configurations containing a symbol from $Q=\{\lfast,\rfast,\lslow,\rslow\}$ in at most one position. We will interpret elements of $Q$ as particles going in different directions at different speeds and which bounce between walls denoted by $\wall$. Let $S:Y\to Y$ be the reversible radius-$2$ CA which does not move occurrences of $\wall$ and which moves $\lfast$ (resp. $\rfast$, $\lslow$, $\rslow$) to the left at speed $2$ (resp. to the right at speed $2$, to the left at speed $1$, to the right at speed $1$) with the additional condition that when an arrow meets a wall, it changes into the arrow with the same speed and opposing direction. More precisely, $S$ is the CA with memory $2$ and anticipation $2$ determined by the local rule $f:A_1^5\to A_1$ defined as follows (where $*$ denotes arbitrary symbols):
\begin{flalign*}
\begin{array}{l l l}
f(\rfast,0,0,*,*)=\rfast	& \quad & f(*,\rslow,0,*,*)=\rslow \\
f(*,\rfast,0,\wall,*)=\lfast & \quad & f(*,*,\rslow,0,*) = 0, \\
f(*,*,\rfast,0,*) = 0 & \quad & f(*,*,\rslow,\wall,*) =\lslow,\\
f(*,0,\rfast,\wall,*)=0 & & \\
f(*,\wall,\rfast,\wall,*)=\rfast & & \\
f(*,*,0,\rfast,\wall)=\lfast & &
\end{array}
\end{flalign*}
with symmetric definitions for arrows in the opposite directions at reflected positions and $f(*,*,a,*,*)=a$ ($a\in A_1$) otherwise. Then let $X=Y\times A_2^\Z$ and $\pi_1:X\to Y$, $\pi_2:X\to A_2^\Z$ be the natural projections $\pi_i(x_1,x_2)=x_i$ for $x_1\in Y,x_2\in A_2^\Z$ and $i\in\{1,2\}$.

Let $x_1\in Y$ and $x_2\in A_2^\Z$ be arbitrary. We define reversible CA $G_2,F_1:X\to X$ by $G_2(x_1,x_2)=(x_1,G^{10}(x_2))$, $F_1(x_1,x_2)=(S(x_1),x_2)$. Additionally, let $F_2:X\to X$ be the involution which maps $(x_1,x_2)$ as follows: $F_2$ replaces an occurrence of $\rfast 0\in A_1^2$ in $x_1$ at a coordinate $i\in\Z$ by an occurrence of $\lslow\wall\in A_1^2$ (and vice versa) \emph{if and only if}
\begin{flalign*}
& G^j(x_2)[i]\notin B\mbox{ for some }0\leq j\leq 5 \\
\mbox{or } &G^j(x_2)[i+1]\notin B \mbox{ for some } 5\leq j\leq 10,
\end{flalign*}
and otherwise $F_2$ makes no changes. Finally, define $F=F_1\circ G_2\circ F_2:X\to X$. The reversible CA $F$ works as follows. Typically particles from $Q$ move in the upper layer in the intuitive manner indicated by the map $S$ and the lower layer is transformed according to the map $G^{10}$. There are some exceptions to the usual particle movements: If there is a particle $\rfast$ which does not have a wall immediately at the front and $x_2$ does not satisfy a local immortality condition in the next $10$ time steps, then $\rfast$ changes into $\lslow$ and at the same time leaves behind a wall segment $\wall$. Conversely, if there is a particle $\lslow$ to the left of the wall $\wall$ and $x_2$ does not satisfy a local immortality condition, $\lslow$ changes into $\rfast$ and removes the wall segment.

We will show that $\lambda^{+}(F)=2$ if $G$ is $(1,5)$-locally immortal with respect to $B$ and $\lambda^{+}(F)\leq \frac{5}{3}$ otherwise. Intuitively the reason for this is that if $x,y\in X$ are two configurations that differ only to the left of the origin, then the difference between $F^i(x)$ and $F^i(y)$ can propagate to the right at speed $2$ only via an arrow $\rfast$ that travels on top of a $(1,5)$-witness. Otherwise, a signal that attempts to travel to the right at speed $2$ is interrupted at bounded time intervals and forced to return at a slower speed beyond the origin before being able to continue its journey to the right. We will give more details.

Assume first that $G$ is $(1,5)$-locally immortal with respect to $B$. Let $x_2\in A_2^\Z$ be a $(1,5)$-witness and define $x_1\in Y$ by $x_1[0]=\rfast$ and $x_1[i]=0$ for $i\neq 0$. Let $x=(0^\Z,x_2)\in X$ and $y=(x_1,x_2)\in X$. It follows that $\pi_1(F^i(x))[2i]=0$ and $\pi_1(F^i(y))[2i]=\rfast$ for every $i\in\N$, so $\lambda^{+}(F)\geq 2$. On the other hand, $F$ has memory $2$ so necessarily $\lambda^{+}(F)=2$.

Assume then that there are no $(1,5)$-witnesses for $G$. Let us denote
\[C(n)=\{x\in A_2^\Z\mid G^{5i+j}(x)[i]\in B\mbox{ for } 0\leq i\leq n,0\leq j\leq 5\}\mbox{ for } n\in\N.\]
Since there are no $(1,5)$-witnesses, by a compactness argument we may fix some $N\in\Npos$ such that $C(2N)=\emptyset$. We claim that $\lambda^{+}(F)\leq \frac{5}{3}$, so let us assume that $(x^{(n)})_{n\in\N}$ with $x^{(n)}=(x_1^{(n)},x_2^{(n)})\in X$ is a sequence of configurations such that $\Lambda_n^{+}(x^{(n)},F)=s_n n$ where $(s_n)_{n\in\N}$ tends to $\lambda^{+}$. There exist $y^{(n)}=(y_1^{(n)},y_2^{(n)})\in X$ such that $x^{(n)}[i]=y^{(n)}[i]$ for $i>-s_n n$ and $F^{t_n}(x)[i_n]\neq F^{t_n}(y)[i_n]$ for some $0\leq t_n\leq n$ and $i_n\geq 0$.

First assume that there are arbitrarily large $n\in\N$ for which $x_1^{(n)}[i]\in \{0,\wall\}$ for $i>-s_n n$ and consider the subsequence of such configurations $x^{(n)}$ (starting with sufficiently large $n$). Since $G$ is a one-sided CA, it follows that $\pi_2(F^{t_n}(x^{(n)}))[j]=\pi_2(F^{t_n}(y^{(n)}))[j]$ for $j\geq 0$. Therefore the difference between $x^{(n)}$ and $y^{(n)}$ can propagate to the right only via an arrow from $Q$, so without loss of generality (by swapping $x^{(n)}$ and $y^{(n)}$ if necessary) $\pi_1(F^{t_n}(x^{(n)}))[j_n]\in Q$ for some $0\leq t_n\leq n$ and $j_n\geq i_n-1$. Fix some such $t_n,j_n$ and let $w_n\in Q^{t_n+1}$ be such that $w_n(i)$ is the unique state from $Q$ in the configuration $F^i(x^{(n)})$ for $0\leq i\leq t_n$. The word $w_n$ has a factorization of the form $w_n=u(v_1 u_1\cdots v_k u_k)v$ ($k\in\N$) where $v_i\in\{\rfast\}^+$, $v\in\{\rfast\}^*$ and $u_i\in(Q\setminus\{\rfast\})^+$, $u\in (Q\setminus\{\rfast\})^*$. By the choice of $N$ it follows that all $v_i,v$ have length at most $N$ and by the definition of the CA $F$ it is easy to see that each $u_i$ contains at least $2(\abs{v_i}-1)+1$ occurrences of $\lslow$ and at least $2(\abs{v_i}-1)+1$ occurrences of $\rslow$ (after $\rfast$ turns into $\lslow$, it must return to the nearest wall to the left and back and at least once more turn into $\lslow$ before turning back into $\rfast$. If $\rfast$ were to turn into $\lfast$ instead, it would signify an impassable wall on the right). If we denote by $k_n$ the number of occurrences of $\rfast$ in $w_n$, then $k_n\leq \abs{w_n}/3+\ord(1)$ (this upper bound is achieved by assuming that $\abs{v_i}=1$ for every $i$) and
\[s_n n\leq \abs{w_n}+2k_n\leq \abs{w_n}+\frac{2}{3}\abs{w_n}+\ord(1)\leq\frac{5}{3}n+\ord(1).\]
After dividing this inequality by $n$ and passing to the limit we find that $\lambda^{+}(F)\leq\frac{5}{3}$.\footnote{By performing more careful estimates it can be shown that $\lambda^{+}(F)=1$, but we will not attempt to formalize the argument for this.}

Next assume that there are arbitrarily large $n\in\N$ for which $x_1^{(n)}[i]\in Q$ for some $i>-s_n n$. The difference between $x^{(n)}$ and $y^{(n)}$ can propagate to the right only after the element from $Q$ in $x^{(n)}$ reaches the coordinate $-s_n n$, so without loss of generality there are $0<t_{n,1}<t_{n,2}\leq n$ and $i_n\geq 0$ such that $\pi_1(F^{t_{n,1}}(x^{(n)}))[-s]\in Q$ for some $s\geq s_n n$ and $\pi_1(F^{t_{n,2}}(x^{(n)}))[i_n]\in Q$. From this the contradiction follows in the same way as in the previous paragraph.
\end{proof}

We are ready to prove the result for CA on full shifts.

\begin{theorem}For reversible CA $F:A^\Z\to A^\Z$ such that $\lambda^{+}(F)\in [0,\frac{5}{3}]\cup\{2\}$ it is undecidable whether $\lambda^{+}(F)\leq \frac{5}{3}$ or $\lambda^{+}(F)=2$.\end{theorem}
\begin{proof}
Let $G:A_2^\Z\to A_2^\Z$, $A_1$, $F=F_1\circ G_2\circ F_2:X\to X$, etc. be as in the proof of the previous theorem. We will adapt the conveyor belt construction from \cite{GS17} to define a CA $F'$ on a full shift which simulates $F$ and has the same right Lyapunov exponent as $F$.

Denote $Q=\{\lfast,\rfast,\lslow,\rslow\}$, $\Sigma=\{0,\wall\}$, $\Delta=\{-,0,+\}$, define the alphabets
\[\Gamma=(\Sigma^2\times\{+,-\})\cup(Q\times \Sigma\times\{0\})\cup(\Sigma\times Q\times\{0\})\subseteq A_1\times A_1\times\Delta\]
and $A=\Gamma\times A_2$ and let $\pi_{1,1},\pi_{1,2}:A^\Z \to A_1^\Z$, $\pi_\Delta:A^\Z\to\Delta^\Z$, $\pi_2:A^\Z\to A_2^\Z$ be the natural projections $\pi_{1,1}(x)=x_{1,1}$, $\pi_{1,2}(x)=x_{1,2}$, $\pi_\Delta(x)=x_\Delta$, $\pi_2(x)=x_2$ for $x=(x_{1,1},x_{1,2},x_\Delta,x_2)\in A^\Z\subseteq (A_1\times A_1\times \Delta\times A_2)^\Z$. For arbitrary $x=(x_1,x_2)\in(\Gamma\times A_2)^\Z$ define $G_2':A^\Z\to A^\Z$ by $G_2'(x)=(x_1,G^{10}(x_2))$.

Next we define $F_1':A^\Z\to A^\Z$. Every element $x=(x_1,x_2)\in(\Gamma\times A_2)^\Z$ has a unique decomposition of the form
\[(x_1,x_2)=\cdots(w_{-2},v_{-2})(w_{-1},v_{-1})(w_{0},v_{0})\allowbreak(w_{1},v_{1})(w_{2},v_{2})\cdots\]
where
\begin{flalign*}
w_i\in &(\Sigma^2\times\{+\})^*((Q\times \Sigma\times\{0\})\cup(\Sigma\times Q\times\{0\}))(\Sigma^2\times\{-\})^* \\
&\cup(\Sigma^2\times\{+\})^*(\Sigma^2\times\{-\})^*
\end{flalign*}
with the possible exception of the leftmost $w_i$ beginning or the rightmost $w_i$ ending with an infinite sequence from $\Sigma^2\times\{+,-\}$.

Let $(c_i,e_i)\in (\Sigma\times \Sigma)^*((Q\times \Sigma)\cup (\Sigma\times Q))(\Sigma\times \Sigma)^*\cup(\Sigma\times \Sigma)^*$ be the word that is derived from $w_i$ by removing the symbols from $\Delta$. The pair $(c_i,e_i)$ can be seen as a conveyor belt by gluing the beginning of $c_i$ to the beginning of $e_i$ and the end of $c_i$ to the end of $e_i$. The map $F_1'$ will shift arrows like the map $F_1$, and at the junction points of $c_i$ and $e_i$ the arrow can turn around to the opposite side of the belt. More precisely, define the permutation $\rho:A_1\to A_1$ by

\begin{flalign*}
\begin{array}{l l l l l l l}
\rho(0)=0 & \quad & \rho(\wall)=\wall & & & & \\
\rho(\lfast)=\rfast & \quad & \rho(\rfast)=\lfast & \quad & \rho(\lslow)=\rslow & \quad & \rho(\rslow)=\lslow
\end{array}
\end{flalign*}
and for a word $u\in A_1^*$ let $\rho(u)$ denote the coordinatewise application of $\rho$. For any word $w=w[1]\cdots w[n]$ define its reversal by $w^R[i]=w[n+1-i]$ for $1\leq i\leq n$. Then consider the periodic configuration $y=[(c_i,v_i)(\rho(e_i),v_i)^R]^\Z\in(A_1\times A_2)^\Z$. The map $F_1:X\to X$ extends naturally to configurations of the form $y$: $y$ can contain infinitely many arrows, but they all point in the same direction and occur in identical contexts. By applying $F_1$ to $y$ we get a new configuration of the form $[(c_i',v_i)(\rho(e_i'),v_i)^R]$. From this we extract the pair $(c'_i,e'_i)$, and by adding plusses and minuses to the left and right of the arrow (or in the same coordinates as in $(c_i,e_i)$ if there is no occurrence of an arrow) we get a word $w_i'$ which is of the same form as $w_i$. We define $F_1':A^\Z\to A^\Z$ by $F_1'(x)=x'$ where $x'=\cdots(w_{-2}',v_{-2})(w_{-1}',v_{-1})(w_{0}',v_{0})\allowbreak(w_{1}',v_{1})(w_{2}',v_{2})\cdots$. Clearly $F_1'$ is shift invariant, continuous and reversible. 

We define the involution $F_2':A^\Z\to A^\Z$ as follows. For $x\in A^\Z$ and $j\in\{1,2\}$ $F_2'$ replaces an occurrence of $\rfast 0$ in $\pi_{1,j}(x)$ at coordinate $i\in\Z$ by an occurrence of $\lslow\wall$ (and vice versa) \emph{if and only if} $\pi_\Delta(x)[i+1]=-$ and

\begin{flalign*}
&G^j(\pi_2(x))[i]\notin B\mbox{ for some }0\leq j\leq 5 \\
\mbox{or } &G^j(\pi_2(x))[i+1]\notin B \mbox{ for some } 5\leq j\leq 10,
\end{flalign*}
and otherwise $F_2$ makes no changes. $F_2'$ simulates the map $F_2$ and we check the condition $\pi_\Delta(x)[i+1]=-$ to ensure that $F_2'$ does not transfer information between neighboring conveyor belts.

Finally, we define $F'=F_1'\circ G_2'\circ F_2':A^\Z\to A^\Z$. The reversible CA $F'$ simulates $F:X\to X$ simultaneously on two layers and it has the same right Lyapunov exponent as $F$. 
\end{proof}

The following corollary is immediate.

\begin{corollary}There is no algorithm that, given a reversible CA $F:A^\Z\to A^\Z$ and a rational number $\epsilon>0$, returns the Lyapunov exponent $\lambda^{+}(F)$ within precision $\epsilon$. \end{corollary}

\section{Lyapunov Exponents of Multiplication \\Automata}

In this section we present a class of multiplication automata which perform multiplication by nonnegative numbers in some integer base. After the definitions and preliminary lemmas we compute their average Lyapunov exponents.

For this section denote $\digs_n=\{0,1,\dots,n-1\}$ for $n\in\N$, $n>1$. To perform multiplication using a CA we need be able to represent a nonnegative real number as a configuration in $\digs_n^{\Z}$. If $\xi\geq0$ is a real number and $\xi=\sum_{i=-\infty}^{\infty}{\xi_i n^i}$ is the unique base-$n$ expansion of $\xi$ such that $\xi_i\neq n-1$ for infinitely many $i<0$, we define $\config_n(\xi)\in \digs_n^{\Z}$ by
\[\config_n(\xi)[i]=\xi_{-i}\]
for all $i\in\Z$. In reverse, whenever $x\in \digs_n^{\Z}$ is such that $x[i]=0$ for all sufficiently small $i$, we define
\[\real_n(x)=\sum_{i=-\infty}^{\infty}{x[-i] n^i}.\]
For words $w=w[1]w[2]\cdots w[k]\in \digs_n^k$ we define analogously
\[\real_n(w)=\sum_{i=1}^{k}{w[i]n^{-i}}.\]
Clearly $\real_n(\config_n(\xi))=\xi$ and $\config_n(\real_n(x))=x$ for every $\xi\geq0$ and every $x\in \digs_n^{\Z}$ such that $x[i]=0$ for all sufficiently small $i$ and $x[i]\neq n-1$ for infinitely many $i>0$.

The fractional part of a number $\xi\in\R$ is
\[\fractional(\xi)=\xi-\lfloor \xi\rfloor\in[0,1).\]

For integers $p,n\geq2$ where $p$ divides $n$ let $\mul_{p,n}:\digs_{n}\times \digs_{n}\to \digs_{n}$ be defined as follows. Let $q$ be such that $pq=n$. Digits $a,b\in \digs_{pq}$ are represented as $a=a_1q+a_0$ and $b=b_1q+b_0$, where $a_0,b_0\in \digs_q$ and $a_1,b_1\in \digs_p$: such representations always exist and they are unique. Then
\[\mul_{p,n}(a,b)=\mul_{p,n}(a_1q+a_0,b_1q+b_0)=a_0p+b_1.\]
An example in the particular case $(p,n)=(3,6)$ is given in Figure \ref{taul}.

\begin{figure}[h]
\centering
\begin{tabular} {c | c c c c c c}
$a\backslash b$ & 0 & 1 & 2 & 3 & 4 & 5 \\ \hline
0 & 0 & 0 & 1 & 1 & 2 & 2 \\
1 & 3 & 3 & 4 & 4 & 5 & 5 \\
2 & 0 & 0 & 1 & 1 & 2 & 2 \\
3 & 3 & 3 & 4 & 4 & 5 & 5 \\
4 & 0 & 0 & 1 & 1 & 2 & 2 \\
5 & 3 & 3 & 4 & 4 & 5 & 5 \\
\end{tabular}
\caption{The values of $\mul_{3,6}(a,b)$.}
\label{taul}
\end{figure}

We define the CA $\Mul_{p,n}:\digs_{n}^{\Z}\to \digs_{n}^{\Z}$ by $\Mul_{p,n}(x)[i]=\mul_{p,n}(x[i],x[i+1])$, so $\Mul_{p,n}$ has memory $0$ and anticipation $1$. The CA $\Mul_{p,n}$ performs multiplication by $p$ in base $n$ in the sense of the following lemma.

\begin{lemma}\label{vastaavuus}$\real_{n}(\Mul_{p,n}(\config_{n}(\xi)))=p\xi$ for all $\xi\geq 0$.\end{lemma}

We omit the proof of the lemma, which can be found for example in~\cite{Kari12b}. The idea of the proof is to notice that the local rule $\mul_{p,n}$ mimics the usual school multiplication algorithm. Because $p$ divides $n$, the carry digits cannot propagate arbitrarily far to the left.

For the statement of the following lemmas, which were originally proved in~\cite{KK17}, we define a function $\integ:\digs_{pq}^+\to\N$ by
\[\integ(w[1]w[2]\cdots w[k])=\sum_{i=0}^{k-1}w[k-i](pq)^i,\]
i.e. $\integ(w)$ is the integer having $w$ as a base-$pq$ representation.

\begin{lemma}Let $w_1,w_2\in \digs_{pq}^k$ for some $k\geq 2$ and let $t>0$ be a natural number. Then
\begin{enumerate}
\item $\integ(w_1)<q^t \implies \integ(\mul_{p,pq}(w_1))<q^{t-1}$ and
\item $\integ(w_2)\equiv \integ(w_1)+q^t\pmod{(pq)^k} \\ \implies \integ(\mul_{p,pq}(w_2))\equiv \integ(\mul_{p,pq}(w_1))+q^{t-1} \pmod{(pq)^{k-1}}$.
\end{enumerate}
\end{lemma}
\begin{proof}
Let $x_i\in \digs_{pq}^\Z$ ($i=1,2$) be such that $x_i[-(k-1),0]=w_i$ and $x_i[j]=0$ for $j<-(k-1)$ and $j>0$. From this definition of $x_i$ it follows that $\integ(w_i)=\real_{pq}(x_i)$. Denote $y_i=\Mul_{p,pq}(x_i)$. We have
\[\sum_{j=-\infty}^{\infty}y_i[-j](pq)^j=\real_{pq}(y_i)=p\real_{pq}(x_i)=p\integ(w_i)\]
and
\begin{flalign*}
\integ(\mul_{p,pq}(w_i))&=\integ(y_i[-(k-1),-1]) \\
&=\sum_{j=1}^{k-1}y_i[-j](pq)^{j-1}\equiv\lfloor \integ(w_i)/q\rfloor\pmod{(pq)^{k-1}}.
\end{flalign*}
Also note that $\integ(\mul_{p,pq}(w_i))<(pq)^{k-1}$.

For the proof of the first part, assume that $\integ(w_1)<q^t$. Combining this with the observations above yields $\integ(\mul_{p,pq}(w_1))\leq\lfloor \integ(w_1)/q\rfloor<q^{t-1}$.

\begin{sloppypar}
For the proof of the second part, assume that $\integ(w_2)\equiv \integ(w_1)+q^t\pmod{(pq)^k}$. Then there exists $n\in\Z$ such that $\integ(w_2)=\integ(w_1)+q^t+n(pq)^k$ and
\end{sloppypar}
\begin{flalign*}
\integ(\mul_{p,pq}(w_2))&\equiv \lfloor \integ(w_2)/q\rfloor\equiv\lfloor \integ(w_1)/q\rfloor+q^{t-1}+np(pq)^{k-1} \\
&\equiv\lfloor \integ(w_1)/q\rfloor+q^{t-1}\equiv \integ(\mul_{p,pq}(w_1))+q^{t-1}\pmod{(pq)^{k-1}}.
\end{flalign*}
\end{proof}

\begin{lemma}\label{godometer}Let $t>0$ and $w_1,w_2\in \digs_{pq}^k$ for some $k\geq t+1$.
\begin{enumerate}
\item If $\integ(w_1)<q^{t}$, then $\integ(\mul_{p,pq}^t(w_1))=0$.
\item If $\integ(w_2)\equiv \integ(w_1)+q^{t}\pmod{(pq)^k}$, then \\ $\integ(\mul_{p,pq}^t(w_2))\equiv \integ(\mul_{p,pq}^t(w_1))+1 \pmod{(pq)^{k-t}}$.
\end{enumerate}
\end{lemma}
\begin{proof}                                                       
Both claims follow by repeated application of the previous lemma.
\end{proof}

The content of Lemma \ref{godometer} is as follows. Assume that $\{w_i\}_{i=0}^{(pq)^k-1}$ is the enumeration of all the words in $\digs_{pq}^k$ in the lexicographical order, meaning that $w_0=00\cdots 00$, $w_1=00\cdots 01$, $w_2=00\cdots 02$ and so on. Then let $i$ run through all the integers between $0$ and $(pq)^k-1$. For the first $q^{t}$ values of $i$ we have $\mul_{p,pq}^t(w_i)=00\cdots 00$, for the next $q^{t}$ values of $i$ we have $\mul_{p,pq}^t(w_i)=00\cdots 01$, and for the following $q^{t}$ values of $i$ we have $\mul_{p,pq}^t(w_i)=00\cdots 02$. Eventually, as $i$ is incremented from $q^{t}(pq)^{k-t}-1$ to $q^{t}(pq)^{k-t}$, the word $\mul_{p,pq}^t(w_i)$ loops from $(pq-1)(pq-1)\cdots (pq-1)(pq-1)$ back to $00\cdots 00$.

From now on let $p,q>1$ be coprime integers. We consider the Lyapunov exponents of the multiplication automaton $\Mul_{p,pq}$. Since $\Mul_{p,pq}$ has memory $0$ and anticipation $1$, it is easy to see that for any $x\in \digs_{pq}^\Z$ we must have $\lambda^{+}(x)=0$ and $\lambda^{-}(x)\leq 1$ and therefore $\lambda^{+}(\Mul_{p,pq})=0$, $\lambda^{-}(\Mul_{p,pq})\leq 1$.

Now consider a positive integer $m>0$. Multiplying $m$ by $p^n$ yields a number whose base-$pq$ representation has length approximately equal to $\log_{pq}(mp^n)=n(\log_{pq}p)+\log_{pq} m$. By translating this observation to the configuration space $\digs_{pq}^\Z$ it follows that $\lambda^{-}(0^\Z,\Mul_{p,pq})=\log_{pq}p$. One might be tempted to conclude from this that $\lambda^{-}(\Mul_{p,pq})=\log_{pq}p$. It turns out that this conclusion is not true.

\begin{theorem}
For coprime $p,q>1$ there is a configuration $x\in\digs_{pq}^\Z$ such that $\lambda^{-}(x,\Mul_{p,pq})=1$. In particular $\lambda^{-}(\Mul_{p,pq})=1$.
\end{theorem}
\begin{proof}
For every $n\in\Npos$ define $x_n=\config_{pq}(q^n-1)$ and $y_n=\config_{pq}(q^n)$. By Lemma \ref{vastaavuus}, $\real(\Mul_{p,pq}^n(x_n))=p^n(q^n-1)<(pq)^n$ and $\real(\Mul_{p,pq}^n(y_n))=p^n q^n=(pq)^n$, which means that $\Mul_{p,pq}^n(x_n)[-n]=0$ and $\Mul_{p,pq}^n(y_n)[-n]=1$. Since $\Mul_{p,pq}$ has memory $0$ and anticipation $1$, it follows that $\Mul_{p,pq}^i(x_n)[-i]\neq \Mul_{p,pq}^i(y_n)[-i]$ when $0\leq i\leq n$ (note that $q^n$ isn't divisible by $pq$ for any $n\in\Npos$, which means that $x_n$ and $y_n$ differ only at the origin). Then choose $x,y\in \digs_{pq}^\Z$ such that $(x,y)\in \digs_{pq}^\Z\times \digs_{pq}^\Z$ is the limit of some converging subsequence of $((x_n,y_n))_{n\in\Npos}$. Then $x$ and $y$ differ only at the origin and $\Mul_{p,pq}^i(x)[-i]\neq \Mul_{p,pq}^i(y)[-i]$ for all $i\in\N$. It follows that $\lambda^{-}(x,\Mul_{p,pq})=1$.
\end{proof}

The intuition that the left Lyapunov exponent of $\Mul_{p,pq}$ ``should be'' equal to $\log_{pq}p$ is explained by the following computation of the average Lyapunov exponent.

\begin{theorem}\label{mulmeasLyap}For coprime $p,q>1$ we have $I_{\mu}^{-}(\Mul_{p,pq})=\log_{pq}p$, where $\mu$ is the uniform measure on $\digs_{pq}^\Z$.\end{theorem}
\begin{proof}
First note that for any $n\in\Npos$ and any $w\in \digs^{n+1}$ the equality $\Lambda_n^{-}(x)=\Lambda_n^{-}(y)$ holds for each pair $x,y\in\cyl(w,0)$, so we may define the quantity $\Lambda_n^{-}(w)=\Lambda_n^{-}(x)$ for $x\in\cyl(w,0)$. For any $i\in\N$ denote $(\Lambda_n^{-})^{-1}(i)=\{x\in\digs_{pq}^\Z\mid \Lambda_n^{-}(x)=i\}$. Then, note that always $\Lambda_n^{-}(x)\leq n$ and define for $0\leq i\leq n$
\[P_n(i)=\{w\in \digs_{pq}^{n+1}\mid \Lambda_n^{-}(w)=i\}\]
which form a partition of $\digs_{pq}^{n+1}$.
From these definitions it follows that
\[I_{n,\mu}^{-}=\int_{x\in \digs_{pq}^\Z}\Lambda_n^{-}(x) d\mu=\sum_{i=0}^{\infty} i\mu((\Lambda_n^{-})^{-1}(i))=(pq)^{-(n+1)}\sum_{i=0}^{n} i\abs{P_n(i)}.\]
To compute $\abs{P_n(i)}$ we define an auxiliary quantity
\[p_n(i)=\{w\in \digs_{pq}^{n+1}\mid i\leq \Lambda_n^{-}(w)\leq n\}:\]
then clearly $P_n(n)=p_n(n)$ and $P_n(i)=p_n(i)\setminus p_n(i+1)$ for $0\leq i<n$. Note that $w\in p_n(i)$ $(0\leq i\leq n)$ is equivalent to the existence of words $u\in \digs_{pq}^i$, $v_1,v_2\in \digs_{pq}^{n+1-i}$ such that $w=uv_1$ and $\mul_{p,pq}^t(uv_1)[1]\neq\mul_{p,pq}^t(uv_2)[1]$ for some $i\leq t\leq n$. By denoting
\[d_n(i)=\{u\in \digs_{pq}^{i}\mid \exists v_1,v_2\in A^{n+1-i},t\in[i,n]:\mul_{p,pq}^t(uv_1)[1]\neq \mul_{p,pq}^t(uv_2)[1]\},\]
it follows that $\abs{p_n(i)}=(pq)^{n+1-i}\abs{d_n(i)}$. By Lemma \ref{godometer}, for a word $u\in \digs_{pq}^i$ the condition $u\in d_n(i)$ is equivalent to the existence of a number divisible by $q^t$ on the open interval $J(u)_t=(\integ(u)(pq)^{t+1-i},(\integ(u)+1)(pq)^{t+1-i})$ for some $t\in[i,n]$. Furthermore, if an integer $m$ is divisible by $q^t$ and $m\in J(u)_t$, then $m(pq)^{n-t}\in J(u)_n$ is divisible by $q^n$. Thus it is sufficient to consider only the interval $J(u)_n$. We use this to compute $\abs{d_n(i)}$.

In the case $(pq)^{n+1-i}>q^n$ (equivalently: $n\log_{pq}q+i<n+1$) each interval $J(u)_n$ contains a number divisible by $q^n$ and therefore $\abs{d_n(i)}=(pq)^i$.

In the case $(pq)^{n+1-i}<q^n$ (equivalently: $n\log_{pq}q+i>n+1$) each interval $J(u)_n$ contains at most one number divisible by
$q^n$. Then $\abs{d_n(i)}$ equals the number of elements on the interval $[0,(pq)^{n+1})$ which are divisible by $q^n$ but not divisible by $(pq)^{n+1-i}$. Divisibility by both $q^n$ and $(pq)^{n+1-i}$ is equivalent to divisibility by $q^n p^{n+1-i}$ because $p$ and $q$ are coprime. Therefore $\abs{d_n(i)}=(pq)^{n+1}/q^n-(pq)^{n+1}/(q^n p^{n+1-i})=(pq)p^n-qp^i$.

Let us denote $\kappa=\left\lfloor n-n\log_{pq}q+1\right\rfloor$. We can see that when $i<\kappa$,
\begin{flalign*}
\abs{P_n(i)}&=\abs{p_n(i)}-\abs{p_n(i+1)}=(pq)^{n+1-i}\abs{d_n(i)}-(pq)^{n-i}\abs{d_n(i+1)} \\
&=(pq)^{n+1}-(pq)^{n+1}=0.
\end{flalign*}
We may compute

\begin{flalign*}
&(pq)^{n+1}I_{n,\mu}^{-}=\sum_{i=0}^{\kappa-1} i\abs{P_n(i)}+ \sum_{i=\kappa}^{n} i\abs{P_n(i)} \\
&=n\abs{p_n(n)}+\sum_{i=\kappa}^{n-1}i(\abs{p_n(i)}-\abs{p_n(i+1)})=\kappa\abs{p_n(\kappa)}+\sum_{i=\kappa+1}^{n}\abs{p_n(i)},
\end{flalign*}
in which
\[\kappa\abs{p_n(\kappa)}=\kappa(pq)^{n+1-\kappa}\abs{d_n(\kappa)}=\kappa(pq)^{n+1-\kappa}(pq)^\kappa=\kappa(pq)^{n+1}\] 
and
\begin{flalign*}
&\sum_{i=\kappa+1}^{n}\abs{p_n(i)}=\sum_{i=\kappa+1}^{n}(pq)^{n+1-i}\abs{d_n(i)}=\sum_{i=\kappa+1}^{n}(pq)^{n+1-i}((pq)p^n-qp^i) \\
&=(pq)p^n\sum_{i=\kappa+1}^{n}(pq)^{n+1-i}-q(pq)^{n+1}\sum_{i=\kappa+1}^{n}q^{-i}\leq(pq)p^n(pq)^{n-\kappa}\sum_{i=0}^{\infty}(pq)^{-i} \\
&\leq 2(pq)p^n(pq)^{n-(n-n\log_{pq}q+1)+1}\leq 2(pq)p^n(pq)^{\log_{pq}q^n}=2(pq)^{n+1}.
\end{flalign*}

Finally, the left average Lyapunov exponent is

\begin{flalign*}
I_{\mu}^{-}&=\lim_{n\to\infty}\frac{I_{n,\mu}^{-}}{n}=\lim_{n\to\infty}\frac{\kappa\abs{p_n(\kappa)}}{(pq)^{n+1}n}+\lim_{n\to\infty}\frac{\sum_{i=\kappa+1}^{n}\abs{p_n(i)}}{(pq)^{n+1}n}=\lim_{n\to\infty}\frac{\kappa}{n} \\
&=1-\log_{pq}q=\log_{pq}p.
\end{flalign*}

\end{proof}

\begin{remark}
Multiplication automata can also be defined more generally. Denote by $\Mul_{\alpha,n}:\digs_n^\Z\to\digs_n^\Z$ the CA that performs multiplication by $\alpha\in\Rpos$ in base $n\in\N$, $n>1$ (if it exists). A characterization of all admissible pairs $\alpha,n$ can be extracted from the paper~\cite{BHM96}, which considers multiplication automata on one-sided configuration spaces $\digs_n^\N$. We believe that $I_{\mu}^{-}(\Mul_{\alpha,n})=\log_{n}\alpha$ for all $\alpha\geq 1$ and all natural numbers $n>1$ such that $\Mul_{\alpha,n}$ is defined (when $\mu$ is the uniform measure of $\digs_n^\Z$). Replacing the application of Lemma~\ref{godometer} by an application of Lemma~5.7 from \cite{KK17} probably yields the result for $\Mul_{p/q,pq}$ when $p>q>1$ are coprime. A unified approach to cover the general case would be desirable.
\end{remark}

\section{Conclusions}
We have shown that the Lyapunov exponents of a given reversible cellular automaton on a full shift cannot be computed to arbitrary precision. Ultimately this turned out to follow from the fact that the tiling problem for 2-way deterministic Wang tile sets reduces to the problem of computing the Lyapunov exponents of reversible CA. Note that the result does not restrict the size of the alphabet $A$ of the CA $F:A^\Z\to A^\Z$ whose Lyapunov exponents are to be determined. Standard encoding methods might be sufficient to solve the following problem.

\begin{problem}
Is there a fixed full shift $A^\Z$ such that the Lyapunov exponents of a given reversible CA $F:A^\Z\to A^\Z$ cannot be computed to arbitrary precision? Can we choose here $\abs{A}=2$? 
\end{problem}

In our constructions we controlled only the right exponent $\lambda^+(F)$ and let the left exponent $\lambda^-(F)$ vary freely. Controlling both Lyapunov exponents would be necessary to answer the following.

\begin{problem}Is it decidable whether the equality $\lambda^{+}(F) +\lambda^{-}(F)=0$ holds for a given reversible cellular automaton $F:A^\Z\to A^\Z$?\end{problem}

We mentioned in the introduction that there exists a single CA whose topological entropy is an uncomputable number. We ask whether a similar result holds also for the Lyapunov exponents.

\begin{problem}Does there exist a single cellular automaton $F:A^\Z\to A^\Z$ such that $\lambda^+(F)$ is an uncomputable number?\end{problem}

By an application of Fekete's lemma the limit that defines $\lambda^+(F)$ is actually the infimum of a sequence whose elements are easily computable when $F:A^\Z\to A^\Z$ is a CA on a full shift. This yields the natural obstruction that $\lambda^+(F)$ has to be an upper semicomputable number. We are not aware of a cellular automaton on a full shift that has an irrational Lyapunov exponent (see Question 5.7 in \cite{CFK19}), so constructing such a CA (or proving the impossibility of such a construction) should be the first step. This problem has an answer for CA $F:X\to X$ on general subshifts $X$, and furthermore for every real $t\geq 0$ there is a subshift $X_t$ and a reversible CA $F_t:X_t\to X_t$ such that $\lambda^+(F_t)=\lambda^-(F_t)=t$ \cite{Hoch11}.

In the previous section we computed that the average Lyapunov exponent $I_{\mu}^{-}$ of the multiplication automaton $\Mul_{p,pq}$ is equal to $\log_{pq}p$ when $p,q>1$ are coprime integers. This in particular shows that average Lyapunov exponents can be irrational numbers. We do not know whether such examples can be found in earlier literature.

\bibliographystyle{plain}
\bibliography{mybib}{}

\end{document}